\documentclass[11pt]{article}
\usepackage{amsmath}
\usepackage{amsthm}
\usepackage{amsfonts}
\usepackage{amssymb}
\usepackage{graphicx}
\usepackage{xcolor}
\usepackage{tcolorbox}
\usepackage{rotating}
\usepackage{authblk}
\usepackage{url}
\usepackage{fullpage}
\usepackage[linesnumbered,ruled]{algorithm2e}
\usepackage[colorlinks,linkcolor=blue,citecolor=blue,urlcolor=magenta,linktocpage,plainpages=false]{hyperref}
\usepackage{enumitem}
\setlist{noitemsep}
\usepackage{xspace}

\newtheorem{theorem}{Theorem}
\newtheorem{proposition}{Proposition}
\newtheorem{lemma}{Lemma}
\newtheorem{definition}{Definition}
\newtheorem{corollary}{Corollary}
\newtheorem{obs}{Observation}

\newtheorem*{remark}{Remark}

\newcommand{\R}{\mathbb{R}}

\renewcommand{\S}{\mathcal{S}}
\newcommand{\Snk}{\mathcal{S}^{n,k}}
\newcommand{\psd}{\mathcal{S}^n_+}
\newcommand{\ip}[2]{\langle #1, #2\rangle}
\newcommand{\distb}{\overline{\textup{dist}}}
\newcommand{\dist}{\textup{dist}}
\newcommand{\diag}{\textup{diag}}

\newcommand{\closure}{$k$-PSD closure\xspace}
\newcommand{\e}{\varepsilon}
\newcommand{\E}{\mathbb{E}}
\newcommand{\cI}{\mathcal{I}}

\DeclareMathOperator{\var}{Var}
\DeclareMathOperator{\Tr}{Tr}

\newcounter{mynotes}
\newcommand{\overbar}[1]{\mkern 1.5mu\overline{\mkern-1.5mu#1\mkern-1.5mu}\mkern 1.5mu}

\newcommand{\overtilde}[1]{\mkern 1.5mu\widetilde{\mkern-1.5mu#1\mkern-1.5mu}\mkern 1.5mu}

\title{Sparse PSD approximation of the PSD cone
}

\author{Grigoriy Blekherman, Santanu S. Dey, Marco Molinaro, Shengding Sun}

\begin{document}
\maketitle
\begin{abstract}
While semidefinite programming (SDP) problems are polynomially solvable in theory, it is often difficult to solve large SDP instances in practice. One technique to address this issue is to relax the global positive-semidefiniteness (PSD) constraint and only enforce PSD-ness on smaller $k\times k$ principal submatrices --- we call this the \emph{sparse SDP relaxation}. Surprisingly, it has been observed empirically that in some cases this approach appears to produce bounds that are close to the optimal objective function value of the original SDP. In this paper, we formally attempt to compare the strength of the sparse SDP relaxation vis-\`a-vis the original SDP from a theoretical perspective. 

In order to simplify the question, we arrive at a data independent version of it, where we compare the sizes of SDP cone and the \closure, which is the cone of matrices where PSD-ness is enforced on all $k\times k$ principal submatrices. In particular, we investigate the question of how far a matrix of unit Frobenius norm in the \closure can be from the SDP cone. We provide two incomparable upper bounds on this farthest distance as a function of $k$ and $n$. We also provide matching lower bounds, which show that the upper bounds are tight within a constant in different regimes of $k$ and $n$. Other than linear algebra techniques, we extensively use probabilistic methods to arrive at these bounds. One of the lower bounds is obtained by observing a connection between matrices in the \closure and matrices satisfying the restricted isometry property (RIP). 
\end{abstract}
\section{Introduction}

\subsection{Motivation}

Semidefinite programming (SDP) relaxations are an important tool to provide dual bounds for many discrete and continuous non-convex optimization problems~\cite{wolkowicz2012handbook}. These SDP relaxations have the form 
	\begin{eqnarray}\label{eq:SDP}
	\begin{array}{rcl}
	&\textup{min} & \langle C, X\rangle\\
	&\textup{s.t.} & \langle A^i, X\rangle \leq b_i ~~\ \forall i \in \{1, \dots, m\}\\
	&& X \in \psd,
	\end{array}
	\end{eqnarray}
	where $C$ and the $A^i$'s are $n \times n$ matrices, $\ip{M}{N} := \sum_{i,j} M_{ij} N_{ij}$, and $\psd$ denotes the cone of $n\times n$ symmetric positive semidefinite (PSD) matrices: 
	
	$$\psd=\{ X\in\R^{n\times n}\,|\,X=X^T, ~x^{\top} Xx\geq 0, ~\forall x\in\R^n\}. $$

	In practice, it is often computationally challenging to solve large-scale instances of SDPs due to the global PSD constraint $X \in \psd$. One technique to address this issue is to consider a further relaxation that replaces the PSD cone by a larger one $\S \supseteq \psd$. In particular, one can enforce PSD-ness on (some or all) smaller $k\times k$ principal submatrices of $X$, i.e., we consider the problem
	\begin{eqnarray}\label{eq:sparseSDP}
	\begin{array}{rcl}
	&\textup{min} & \langle C, X\rangle\\
	&\textup{s.t.} & \langle A^i, X\rangle \leq b_i \ \forall i \in \{1, \dots, m\}\\
	&& \textup{selected } k \times k \textup{ principal submatrices of }X \in \mathcal{S}^k_+.
	\end{array}
	\end{eqnarray}
 We call such a relaxation the \emph{sparse SDP relaxation}.

 One reason why these relaxations may be solved more efficiently in practice is that we can enforce PSD constraints by iteratively separating linear constraints. Enforcing PSD-ness on smaller $k\times k$ principal submatrices leads to linear constraints that are sparser, an important property leveraged by linear programming solvers that greatly improves their efficiency~\cite{bixby2002solving, walter2014sparsity, amaldi2014coordinated,coleman1990large,reid1982sparsity}. This is an important motivation for using~{sparse SDP} relaxations~\cite{qualizza2012linear, baltean2018selecting,DeyWorking}. (This is also the motivation for studying approximations of polytopes~\cite{dey2015approximating}, convex hulls of integer linear programs~\cite{dey2017analysis,walter2014sparsity,dey2018theoretical}, and integer programming formulations~\cite{hojny2017size} by sparse linear inequalities.) This is our reason for calling the relaxation obtained by enforcing the SDP constraints on smaller $k\times k$ principal submatrices of $X$ as the sparse SDP relaxation.

	It has been observed that sparse SDP relaxations not only can be solved much more efficiently in practice, but in some cases they produce bounds that are close to the optimal value of the original SDP. See~\cite{qualizza2012linear, baltean2018selecting,DeyWorking} for successful applications of this technique for solving box quadratic programming instances, and~\cite{sojoudi2014exactness, kocuk2016strong} for solving the optimal power flow problem in power systems.

	Despite their computational success, theoretical understanding of sparse SDP relaxations remains quite limited. In this paper, we initiate such theoretical investigation. Ideally we would like to compare the objective function values of (\ref{eq:SDP}) and (\ref{eq:sparseSDP}), but this appears to be a very challenging problem. Therefore, we consider a simpler data-independent question, where we ignore the data of the SDP and the particular selected principal submatrices, to arrive at the following:

	\vspace{14pt}

	\mbox{
	\begin{minipage}{0.9\textwidth}
		\it How close to the PSD cone $\psd$ do we get when we only enforce PSD-ness on $k \times k$ principal submatrices? 
	\end{minipage}
	}\\
	\vspace{10pt}

To formalize this question, we begin by defining the \emph{\closure}, namely matrices that satisfy all $k \times k$ principal submatrices PSD constraints. 
\begin{definition}[\closure]
	Given positive integers $n$ and $k$ where $2 \leq k \le n$, the \closure  $\mathcal{S}^{n,k}$ is the set of all $n \times n$ symmetric real matrices where all $k \times k$ principal submatrices are PSD.  
\end{definition}

It is clear that the \closure is a relaxation of the PSD cone (i.e., $\S^{n,k} \supseteq \psd$ for all $2 \leq k \le n$) and is an increasingly better approximation as the parameter $k$ increases, i.e., we enforce that larger chunks of the matrix are PSD (in particular $\S^{n,n} = \psd$). The SOCP relaxation formulated in \cite{sojoudi2014exactness} is equivalent to using the \closure with $k=2$ to approximate the PSD cone. Our definition is a generalization of this construction.
	
It is worth noting that the dual cone of $\mathcal{S}^{n,k}$ is the set of symmetric matrices with factor width $k$, defined and studied in~\cite{boman2005factor}. In particular, the set of symmetric matrices with factor width 2 is the set of scaled diagonally dominant matrices~\cite{wang2019polyhedral}, i.e., symmetric matrices $A$ such that $DAD$ is diagonally dominant for some positive diagonal matrix $D$. Note that~\cite{ahmadi2019dsos} uses scaled diagonally dominant for constructing inner approximation of the SDP cones for use in solving polynomial optimization problems.		

	\subsection{Problem setup}
	
We are interested in understanding how well the \closure approximates the PSD cone for the different values of $k$ and $n$. To measure this approximation we would like to consider the matrix in the \closure that is farthest from the PSD cone. We need to make two choices here: the norm to measure this distance and a normalization method (since otherwise there is no upper bound on the distance between matrices in the PSD cone and the \closure). 

We will use the Frobenius norm $\|\cdot\|_F$ for both purposes. That is, the distance between a matrix $M$ and the PSD cone is measured as $\dist_F(M,\psd) = \textup{inf}_{N \in \S^n_{+}}\|M - N\|_F$, and we restrict our attention to matrices in \closure with Frobenius norm equal to $1$.
Thus we arrive at the \emph{(normalized) Frobenius distance} between the \closure and the PSD cone, namely the largest distance between a unit-norm matrix $M$ in $\S^{n,k}$ and the cone $\psd$:
\begin{align*}
\distb_F(\mathcal{S}^{n,k},\psd)&=\sup_{M\in\mathcal{S}^{n,k},\,\|M\|_F=1}\dist_F(M,\psd)\\
		& = \sup_{M\in\mathcal{S}^{n,k},\,\|M\|_F=1} \inf_{N \in \psd} \|M - N\|_F.
\end{align*}
	Note that since the origin belongs to $\psd$ this distance is at most $1$.
	
The rest of the paper is organized as follows: Section~\ref{sec:res} presents all our results and Section~\ref{sec:con} concludes with some open questions. Then Section~\ref{sec:pre} presents additional notation and background results needed for proving the main results. The remaining sections present the proofs of the main results.

\section{Our results}\label{sec:res}

In order to understand how well the \closure approximates the PSD cone we present:
	
	\begin{itemize}
		\item Matching upper and lower bounds on $\distb_F(\mathcal{S}^{n,k},\psd)$ for different regimes of $k$.
		
		\item Show that 
		{a polynomial number} of $k \times k$ PSD constraints are sufficient to provide a good approximation (in Frobenius distance) to the full \closure (which has ${n \choose k} \approx \big(\frac{en}{k}\big)^k$ such constraints).  
	\end{itemize}

We present these result in more details in the following subsections.

\subsection{Upper bounds}\label{main:upper}

	First we show that the distance between the \closure and the SDP cone is at most roughly $\approx \frac{n-k}{n}$. In particular, this bound approximately goes from $1$ to $0$ as the parameter $k$ goes from $2$ to $n$, as expected. 

\begin{theorem}\label{thm:upper1}
	For all $2\leq k<n$ we have 
	\begin{equation}\label{upper1}
	\distb_F(\mathcal{S}^{n,k},\mathcal{S}^n_+)\leq \frac{n -k}{ n + k - 2}.
	\end{equation}
\end{theorem}
	
	The idea for obtaining this upper bound is the following: given any matrix $M$ in the \closure $\mathcal{S}^{n,k}$, we construct a PSD matrix $\overtilde{M}$ by taking the average of the (PSD) matrices obtained by zeroing out all entries of $M$ but those in a $k \times k$ principal submatrix; the distance between $M$ and $\overtilde{M}$ provides an upper bound on $\distb_F(\mathcal{S}^{n,k},\mathcal{S}^n_+)$. The proof of Theorem~\ref{thm:upper1} is provided in Section~\ref{sec:upper1}.

	\medskip
	It appears that for $k$ close to $n$ this upper bound is not tight. In particular, our next upper bound is of the form $(\frac{n-k}{n})^{3/2}$, showing that the gap between the \closure and the PSD cone goes to 0 as $n-k \rightarrow n$ at a faster rate than that prescribed by the previous theorem. In particular, for $k = n - c$ for a constant $c$, Theorem~\ref{thm:upper1} gives an upper bound of $O\left(\frac{1}{n}\right)$ whereas the next Theorem gives an improved upper bound of $O\left(\frac{1}{n^{3/2}}\right)$. 

\begin{theorem}\label{thm:upper2}
	Assume $n \ge 97$ and $k \ge \frac{3n}{4}$. Then 
	\begin{equation}\label{upper2}	\distb_F(\mathcal{S}^{n,k},\mathcal{S}^n_+) \le 96 \, \bigg(\frac{n-k}{n}\bigg)^{3/2}. 
	\end{equation}
\end{theorem}

It is easy to verify that for sufficiently large $r$  if $k > rn$, then the upper bound given by Theorem~\ref{thm:upper2} dominates the upper bound given by Theorem~\ref{thm:upper1}. 

	The proof of Theorem \ref{thm:upper2} is more involved than that of Theorem \ref{thm:upper1}. The high-level idea is the following: Using Cauchy's Interlace Theorem for eigenvalues of hermitian matrices, we first verify that every matrix in $\mathcal{S}^{n,k}$ has at most $n - k$ negative eigenvalues. Since the PSD cone consists of symmetric matrices with non-negative eigenvalues, it is now straightforward to see that the distance from a unit-norm matrix $M \in \mathcal{S}^{n,k}$ to $\psd$ is upper bounded by the absolute value of the most negative eigenvalue of $M$ times $\sqrt{n -k}$. To bound a negative eigenvalue $-\lambda$ of $M$ (where $\lambda \geq 0$), we consider an associated eigenvector $v \in \R^n$ and  randomly sparsify it to obtain a random vector $V$ that has at most $k$ non-zero entries. By construction we ensure that $V \approx v$, and that $V$ remains almost orthogonal to all other eigenvectors of $M$. This guarantees that $V^{\top} M V \approx -\lambda + \textrm{``small error''}$. On the other hand, since only $k$ entries of $V$ are non-zero, it guarantees that $V^\top M V$ only depends on a $k \times k$ submatrix of $M$, which is PSD by the definition of the \closure; thus, we have $V^\top M V \ge 0$. Combining these observations we get that $\lambda \leq \textrm{``small error''}$. This eigenvalue bound is used to upper bound the distance from $M$ to the PSD cone.  A proof of Theorem~\ref{thm:upper2} is provided in Section~\ref{Sec:proofThm2}.
	

\subsection{Lower bounds}

	We next provide lower bounds on $\distb_F(\mathcal{S}^{n,k},\mathcal{S}^n_+)$ that show that the upper bounds presented in Section~\ref{main:upper} are tight for various regimes of $k$. The first lower bound, presented in the next theorem, is obtained by a simple construction of an explicit matrix in the \closure that is far from being PSD. Its proof is provided in Section~\ref{sec:proofThmlower1}.
	
\begin{theorem}\label{thm:lower1}
	For all $2\leq k<n$, we have 
	\begin{equation}\label{lower1}
\distb_F(\mathcal{S}^{n,k},\mathcal{S}^n_+)\geq\frac{n - k}{\sqrt{ (k-1)^2\, n  + n(n-1)} }.
	\end{equation}
\end{theorem}

Notice that for small values of $k$ the above lower bound is approximately $\approx \frac{n - k}{n}$ which matches the upper bound from Theorem~\ref{thm:upper1}. For very large values of $k$. i.e. $k  = n - c$ for a constant $c$, the above lower bound is approximately $\approx \frac{c}{n^{3/2}}$ which  matches the upper bound by Theorem~\ref{thm:upper2}. 

%
%
%

Now consider the regime where $k$ is a constant fraction of $n$. While our upper bounds give $\distb_F(\Snk,\mathcal{S}^n_+) = O(1)$, Theorem~\ref{thm:lower1} only shows that this distance is at least $\Omega(\frac{1}{\sqrt{n}})$, leaving open the possibility that the \closure provides a sublinear approximation of the PSD cone in this regime. Unfortunately, our next lower bound shows that this is not that case: the upper bounds are tight (up to a constant) in this regime. 

\begin{theorem}\label{thm:lower2}
	Fix a constant $r < \frac{1}{93}$ and let $k = rn$. Then for all $ k\geq 2$, $$\distb_F(\mathcal{S}^{n,k},\mathcal{S}^{n}_+)>\frac{\sqrt{r-93r^2}}{\sqrt{162r+3}},$$ which is independent of $n$. 
\end{theorem}

	For this construction we establish a connection with the \emph{Restricted Isometry Property} (RIP)~\cite{CandesTao1,candes2006stable}, a very important notion in signal processing and recovery~\cite{Candes2008survey,Chartrand2008RIP}.  Roughly speaking, these are matrices that approximately preserve the $\ell_2$ norm of sparse vectors. The details of this connection and the proof of Theorem~\ref{thm:lower2} are provided in Section~\ref{sec:proofThm3}.

\subsection{Achieving the strength of $\mathcal{S}^{n,k}$ by a polynomial number of PSD constraints}

In practice one is unlikely to use the full \closure, since it involves enforcing the PSD-ness for all ${n \choose k} \approx \left(\frac{en}{k}\right)^k$ principal submatrices. Is it possible to achieve the upper bounds mentioned above while enforcing PSD-ness on fewer principal submatrices?  We show that the upper bound given by \eqref{upper1} 
can also be achieved with factor $1+\epsilon$ and probability at least $1-\delta$ by randomly sampling $O\left(\frac{n^2}{\e^2}\ln \frac{n}{\delta}\right)$ of the $k\times k$ principal submatrices. 

\begin{theorem}\label{prop:sampling} Let $2 \leq k \leq n -1$. Consider $\e,\delta>0$ and let $$m :=\frac{12n(n-1)^2}{\e^2 (n-k)^2 k}\ln\frac{2n^2}{\delta} \in O\left(\frac{n^2}{\e^2}\ln \frac{n}{\delta}\right).$$ Let $\cI = (I_1,\ldots,I_m)$ be a sequence of random $k$-sets independently uniformly sampled from ${[n] \choose k}$, and define $\S_{\cI}$ as the set of matrices satisfying the PSD constraints for the principal submatrices indexed by the $I_i$'s, namely
	\begin{align*}
		\S_{\cI} := \{M \in \R^{n \times n} : M_{I_i} \succeq 0,~\forall i \in [m]\}.
	\end{align*}
	Then with probability at least $1 - \delta$ we have 
$$\distb_F(\S_{\cI},\psd) \leq (1+\e)\frac{n -k}{ n + k - 2}.$$
\end{theorem}

\begin{remark}
	Since the zero matrix is PSD, by definition we always have $\distb_F(\S_{\cI},\psd) \leq 1$. So in order for the bound given by Theorem~\ref{prop:sampling} to be of interest, we need $(1+\e)\frac{n -k}{ n + k - 2} \leq 1$, which means $\e \leq \frac{2k-2}{n-k}$. Plugging this into $m$, we see that 
we need {at least $\frac{3n(n-1)^2}{k(k-1)^2}\ln \frac{2n^2}{\delta}=\tilde{O}(\frac{n^3}{k^3})$ samples to obtain a nontrivial upper bound on the distance}. 
\end{remark}

Recall that a collection $\mathcal{D}$ of $k$-sets of $[n]$ (called blocks) is called a $2$-design (also called a balanced incomplete block design or BIBD) if every pair of elements in $[n]$ belongs to the same number of blocks, denoted $\lambda$. It follows that every element of $[n]$ belongs to the same number of blocks, denoted $r$. Let $b$ be the total number of blocks. The following relation is easily shown by double-counting:
\begin{equation*}
\frac{\lambda}{r}=\frac{k-1}{n-1}.
\end{equation*}
For background on block designs we refer to \cite[Chapters 1 and 2]{MR2029249}.
It immediately follows from the discussion in Sections \ref{sec:upper1}  and \ref{sec:propsampling} that the strength of the bound in \eqref{upper1} can be achieved by the blocks of a $2$-design, instead of using all $k \times k$ submatrices. 

It is known from the work of Wilson \cite[Corollary A and B]{MR366695} that, a $2$-design with $b=n(n-1)$ exists for all sufficiently large values of $n$, although to the best of our knowledge no explicit construction is known. (Wilson's theorem gives a much more general statement for existence of $2$-designs). Therefore, for almost all $n$ we can achieve the strength of bound \eqref{upper1} while only using $n(n-1)$ submatrices.

Fisher's inequality states that $b\geq n$, so we need to enforce PSD-ness of at least $n$ minors if we use a $2$-design. A $2$-design is called \textit{symmetric} if $b=n$. 
Bruck-Ryser-Chowla Theorem gives necessary conditions on $b$, $k$ and $\lambda$, for which a symmetric $2$-designs exist, and this is certainly a limited set of parameters.
Nevertheless, symmetric $2$-designs may be of use in practice, as they give us the full strength of \eqref{upper1} while enforcing PSD-ness of only $n$ $k\times k$ minors. Some important examples of symmetric $2$-designs are finite projective planes (symmetric $2$-designs with $\lambda=1$), biplanes ($\lambda=2$) and Hadamard $2$-designs.

\section{Conclusion and open questions}\label{sec:con}
In this paper, we have been able to provide various upper and lower bounds on $\distb_F(\mathcal{S}^{n,k},\mathcal{S}^n_+)$. In two regimes our bounds on $\distb_F(\mathcal{S}^{n,k},\mathcal{S}^n_+)$ are quite tight. These are: (i) $k$ is small, i.e., $2 \leq k \leq \sqrt{n}$ and (ii) $k$ is quite large, i.e., $k = n - c$ where $c$ is a constant. These are shown in the first two rows of Table~\ref{table:one}. When $k/n$ is a constant, we have also established upper and lower bounds on $\distb_F(\mathcal{S}^{n,k},\mathcal{S}^n_+)$ that are independent of $n$. However, our upper and lower bounds are not quite close when viewed as a function of the ratio $k/n$. Improving these bounds  as a function of this ratio is an important open question.

\begin{center}
\begin{table}[h]
\caption{Bounds on $\distb_F(\mathcal{S}^{n,k},\mathcal{S}^n_+)$ for some regimes}
\label{table:one}
	\begin{tabular}{|c|c|c|}
	\hline
		Regime & Upper bound & Lower bound \\
		\hline
		(small k) $ 2 \leq k \leq \sqrt{n}$ & $\frac{n -k}{n }$ (Simplified from Thm~\ref{thm:upper1})&  $\frac{1}{\sqrt{2}}\frac{n -k}{n}$ (Simplified from Thm~\ref{thm:lower1}) \\
		\hline
		(large k) $k \geq n-c$ & $96 \left(\frac{c}{n} \right)^{3/2}$ (Simplified from Thm~\ref{thm:upper2}) & $\frac{1}{\sqrt{2}} \frac{c}{n^{3/2}}$ (Simplified from Thm~\ref{thm:lower1}) \\
		$(n \geq 97, k \geq 0.75n)$ & &  \\
		\hline
		($k/n$ is a constant) $k = rn$ & Constant, independent of $n$ & Constant, independent of $n$ \\
		($r < \frac{1}{93}$) & $1 - r$ (Simplified from Thm~\ref{thm:upper1}) & $ \sqrt{\frac{r - 93r^2}{5}} $ (Simplified from Thm \ref{thm:lower2}) \\
\hline	
	\end{tabular}
\end{table}
\end{center}

We also showed that instead of selecting all minors, only a polynomial number of randomly selected minors realizes upper bound (\ref{upper1}) within factor $1+\e$ with high probability. An important question in this direction is to deterministically and strategically determine principal submatrices to impose PSD-ness, so as to obtain the best possible bound for (\ref{eq:sparseSDP}) 
As discussed earlier, such questions are related to exploring 2-designs and perhaps further generalizations of results presented in \cite{kim2011exploiting}. 


	\section{Notation and Preliminaries}\label{sec:pre}
	The \emph{support} of a vector is the set of its non-zero coordinates, and we call a vector \emph{$k$-sparse} if its support has size at most $k$. 
	We will use $[n]$ to denote the set $\{1,...,n\}$. A \emph{$k$-set} of a set $A$ is a subset $B\subset A$ with $|B|=k$. Given any vector $x\in\R^n$ and a $k$-set $J\subset [n]$ we define $x_J\in\R^k$ as the vector where we remove the coordinates whose indices are not in $J$. Similarly, for a matrix $M \in \mathbb{R}^{n \times n}$ and a $k$-set $J\subset [n]$, we denote the principal submatrix of $M$ corresponding to the rows and columns in $J$ by $M_J$.
	
	\subsection{Linear algebra}
		
	Given any $n\times n$ matrix $A=[a_{ij}]$ its trace (the sum of its diagonal entries) is denoted as $\Tr(A)$. Recall that $\Tr(A)$ is also equal to the sum of all eigenvalues of $A$, counting multiplicities. Given a symmetric matrix $A$, we use $\lambda_1(A) \geq \lambda_2(A) \ge \dots$ to denote its eigenvalues in non-increasing order. 
	
	We remind the reader that, a real symmetric $n\times n$ matrix $M$ is said to be PSD if $x^{\top} Mx\geq 0$ for all $x\in\R^n$, or equivalently  all of its eigenvalues are non-negative. We also use the notation that $A\succeq B$ if $A-B$ is PSD. 

	We next present the famous Cauchy's Interlace Theorem which will be important for obtaining an upper bound on the number of negative eigenvalues of matrices in $\mathcal{S}^{n,k}$. A proof can be found in \cite{Horn1985matrix}. 
	
	\begin{theorem}[Cauchy's Interlace Theorem]\label{thm:Cauchy}
		
		Consider an $n \times n$ symmetric matrix $A$ and let $A_J$ be any of its $k\times k$ principal submatrix. Then for all $1\leq i\leq k$, 
		$$
		\lambda_{n-k+i}(A)\leq \lambda_i (A_J)\leq \lambda_i (A). 
		$$
	\end{theorem}


	

	
	\subsection{Probability}

	These following concentration inequalities will be used throughout, and can be found in \cite{concentration}. 
	
	\begin{theorem}[Markov's Inequality]
		Let $X$ be a non-negative random variable. Then for all $a \ge 1$,
		
		$$
		\Pr(X\ge a\mathbb{E}(X))\le \frac{1}{a}. 
		$$
	\end{theorem}
	
	\begin{theorem}[Chebyshev's Inequality] Let $X$ be a random variable with finite mean and variance. Then for all $a>0$, $$\Pr(|X - \mathbb{E}(X)| \geq a) \leq \frac{\var(X)}{a^2}. $$
	\end{theorem}

	\begin{theorem}[Chernoff Bound]
		Let $X_1,...,X_n$ be i.i.d. Bernoulli random variables, with $\Pr(X_i=1)=\mathbb{E}(X_i)=p$ for all $i$. Let $X=\sum_{i=1}^n X_i$ and $\mu=\mathbb{E}(X)=np$. Then for any $0<\delta<1$, 
		
		$$
		\Pr\big(|X-\mu|>\delta \mu\big)\leq 2\exp\bigg(-\frac{\mu\delta^2}{3}\bigg). 
		$$
	\end{theorem}


\section{Proof of Theorem~\ref{thm:upper1}: Averaging operator}\label{sec:upper1}

	Consider a matrix $M$ in the \closure $\S^{n,k}$ with $\|M\|_F = 1$. To upper bound its distance to the PSD cone we transform $M$ into a ``close by'' PSD matrix $\overtilde{M}$. 
	
	The idea is clear: since all $k \times k$ principal submatrices of $M$ are PSD, we define $\overtilde{M}$ as the average of these minors. More precisely, for a set $I \subseteq [n]$ of $k$ indices, let $M^I$ be the matrix where we zero out all the rows and columns of $M$ except those indexed by indices in $I$; then $\overtilde{M}$ is the average of all such matrices:
	\begin{align*}
		\overtilde{M} := \frac{1}{{n \choose k}} \sum_{I \in {[n] \choose k}} M^I.
	\end{align*} 
	Notice that indeed since the principal submatrix $M_I$ is PSD, $M^I$ is PSD as well: for all vectors $x \in \R^n$, $x^\top M^I x = x_I M_I x_I \ge 0$. 
	Since the average of PSD matrices is also PSD, we have that $\overtilde{M}$ is PSD, as desired.    
	
	Moreover, notice that the entries of $\overtilde{M}$ are just scalings of the entries of $M$, depending on how many terms of the average it is not zeroed out:
	
\begin{enumerate}
\item \textbf{Diagonal terms:} These are scaled by the factor $$\frac{{n \choose k}  - {n -1 \choose k}  }{{n \choose k}} = \frac{k}{n},$$ that is, $\overtilde{M}_{ii} = \frac{k}{n}M_{ii}$ for all $i \in [n]$.

\vspace{6pt}
\item \textbf{Off-diagonal terms:} These are scaled by the factor $$\frac{{n \choose k}  - (2{n - 1 \choose k}  - {n - 2 \choose k}) }{{n \choose k}} = \frac{k(k-1)}{n(n-1)},$$ that is, $\overtilde{M}_{ij} = \frac{k(k-1)}{n(n-1)}M_{ij}$ for all $i \neq j$.
\end{enumerate}

	To even out these factors, we define the scaling $\alpha := \frac{2n(n-1)}{k (n + k -2)}$ and consider $\alpha \overtilde{M}$. Now we have that the difference between $M$ and $\alpha \overtilde{M}$ is a uniform scaling (up to sign) of $M$ itself: $(M - \alpha \overtilde{M})_{ii} = (1 - \alpha \frac{k}{n})\, M_{ii} = -\frac{n -k}{ n + k - 2}\, M_{ii}$, and $(M - \alpha \overtilde{M})_{ij} = (1-\alpha \frac{k(k-1)}{n(n -1)})\,M_{ij} = \frac{n -k}{ n + k - 2}\, M_{ij}$ for $i\neq j$. Therefore, we have 
	\begin{align*}
	\textup{dist}_{F}(M, \psd) ~\leq~ \|M - \alpha \overtilde{M}\|_F ~=~ \frac{n -k}{ n + k - 2}\,\|M\|_F ~=~ \frac{n -k}{ n + k - 2}.
	\end{align*}
	Since this holds for all unit-norm matrix $M \in \S^{n,k}$, this upper bound also holds for $\distb_F(\S^{n,k},\psd)$. This concludes the proof of Theorem \ref{thm:upper1}.


\section{Proof of Theorem~\ref{thm:upper2}: Randomized sparsification}\label{Sec:proofThm2}

	Let $M \in \S^{n,k}$ be a matrix in the \closure with $\|M\|_F = 1$. To prove Theorem \ref{thm:upper2}, we show that the Frobenius distance from $M$ to the PSD cone is at most $O\big((\frac{n-k}{n})^{3/2}\big)$. 
	We assume that $M$ is not PSD, otherwise we are done, and hence it has a negative eigenvalue. We write $M$ in terms of its eigendecomposition: Let $-\lambda_1 \le -\lambda_2 \le \ldots \le -\lambda_\ell$ and $\mu_1,\ldots,\mu_{n-\ell}$ be the negative and non-negative eigenvalues of $M$, and let $v^1,\ldots,v^{\ell} \in \R^n$ and $w^1,\ldots,w^{n-\ell} \in \R^n$ be orthonormal eigenvectors relative to these eigenvalues. Thus 
	\begin{equation}\label{eq:decomp}
	M=- \sum_{i \le \ell} \lambda_i v^i (v^i)^\top + \sum_{i \le n-\ell} \mu_i w^i (w^i)^\top.
	\end{equation}
	Notice that since $\|M\|_F = 1$ we have

	\begin{align}
	\sum_{i \le \ell} \lambda_i^2 + \sum_{i \le n-\ell} \mu_i^2 = 1. \label{eq:normEigen}
	\end{align}
	
	We first relate the distance from $M$ to the PSD cone to its negative eigenvalues. 
	
	\subsection{Distance to PSD cone and negative eigenvalues.}
	
	We start with the following general observation. 
	
	\begin{proposition}\label{prop:neval}
		Suppose $M$ is a symmetric $n\times n$ matrix with $\ell\leq n$ negative eigenvalues. Let $-\lambda_1\leq -\lambda_2\leq ...\leq -\lambda_\ell < 0$ and $\mu_1,...,\mu_{n-l}\geq 0$ be the negative and non-negative eigenvalues of $M$. Then
		$$
		\dist_F(M,\psd)=\sqrt{\sum_{i=1}^\ell \lambda_i^2}. 
		$$
	\end{proposition}
	\begin{proof}
		Let $V$ be the orthonormal matrix that diagonalizes $M$, i.e., $$V^{\top} MV=D :=\diag(-\lambda_1,...,-\lambda_\ell,\mu_1,...,\mu_{n-\ell}).$$ It is well-known that the Frobenius norm is invariant under orthonormal transformation. Therefore, for any $N
		\in\psd$ we have
		$$
		\dist_F(M,N)~=~\|M-N\|_F~=~\|V^{\top }(M-N)V\|_F ~=~ \dist_F(D,\,V^{\top}NV).
		$$ 
		Since $N\in \psd$ iff $V^{\top}NV\in \psd$, we see that $\dist_F(M,\psd)=\dist_F(D,\psd)$. So we only need to show that the latter is $\sqrt{\sum_{i=1}^\ell \lambda_i^2}$. 
		
		Let $D_+=\diag(0,...,0,\mu_1,...,\mu_{n-\ell})$ be obtained from $D$ by making all negative eigenvalues zero. Then $$\|D-D_+\|_F=\sqrt{\sum_{i=1}^n\sum_{i=1}^n (D-D_+)_{ij}^2}=\sqrt{\sum_{i=1}^\ell \lambda_i^2}.$$
		It then suffices to show that $D^+$ is the PSD matrix closest to $D$. For that, let $N$ be any PSD matrix. Then $N_{ii}=e_i^{\top} N e_i\geq 0$ for all $i$, where $e_i$ is the standard unit vector on $i^{th}$ coordinate. Thus we have
		\begin{align*}
		\|D-N\|_F&~=~\sqrt{\sum_{i=1}^\ell (N_{ii}+\lambda_{i})^2+\sum_{i=\ell+1}^{n} (\mu_{i-\ell}-N_{ii})^2+\sum_{i=1}^n\sum_{j\neq i}N_{ij}^2} ~\geq~ \sqrt{\sum_{i=1}^\ell \lambda_i^2}. 
		\end{align*}
		This concludes the proof. 
	\end{proof}
	
	In addition, Cauchy's Interlace Theorem gives an upper bound on the number of negative eigenvalues of matrices in $\Snk$. 
	
	\begin{proposition}\label{cor:nevals}
		Any $A\in \mathcal{S}^{n,k}$ has at most $n-k$ negative eigenvalues. 
	\end{proposition}
	\begin{proof}
		Let $J$ be any $k$-subset of $[n]$. Since $A\in\mathcal{S}^{n,k}$ we have that $A_J$ is PSD, so in particular $\lambda_k({A_J})\geq 0$. Thus, by Theorem~\ref{thm:Cauchy} the original matrix $A$ also has $\lambda_k(A)\geq 0$, and so the first $k$ eigenvalues of $A$ are nonnegative. 
	\end{proof}

	Using Proposition~\ref{prop:neval} and Proposition~\ref{cor:nevals}, given any symmetric matrix $M\in\Snk$ we can get an upper bound on $\dist_F(M,\psd)$ using its smallest eigenvalue. 
	\begin{proposition}\label{prop:ub:smallev}
		Consider a matrix $M\in\Snk$ with smallest eigenvalue $-\lambda_1 < 0$. Then
		$$
		\dist_F(M,\psd)\leq\sqrt{n-k} \cdot\lambda_1. 
		$$
	\end{proposition}

	\begin{proof}
		Letting $-\lambda_1,\ldots,-\lambda_\ell$ be the negative eigenvalues of $M$, we have from Proposition~\ref{prop:neval} that $\dist_F(M,\psd)=\sqrt{\sum_{i=1}^\ell \lambda_i^2} \le \sqrt{\ell}\cdot  \lambda_1$, since $-\lambda_1$ is the smallest eigenvalue. Since $\ell \le n-k$, because of Proposition~\ref{cor:nevals}, we obtain the result.
	\end{proof}
		

	\subsection{Upper bounding $\lambda_1$}

	Given the previous proposition, fix throughout this section a (non PSD) matrix $M \in \Snk$ with smallest eigenvalue $- \lambda_1 < 0$. Our goal is to upper bound $\lambda_1$.

 The first observation is the following: Consider a symmetric matrix $A$ and a set of coordinates $I \subseteq [n]$, and notice that for every vector $x \in \R^n$ supported in $I$ we have $x^\top A x = x_I^\top A_I x_I$. Thus, the principal submatrix $A_I$ is PSD iff for all vectors $x \in \R^n$ supported in $I$ we have $x^\top A x \ge 0$. Applying this to all principal submatrices gives a characterization of the \closure via $k$-sparse test vectors. 
	
	\begin{obs} \label{obs:sparse}
		A symmetric real matrix $A$ belongs to $\S^{n,k}$ iff for all $k$-sparse vectors $x \in \R^n$ we have $x^\top A x \ge 0$. 
	\end{obs}
		
	Using this characterization, and the fact that $M \in \S^{n,k}$, the idea to upper bound $\lambda_1$ is to find a  vector $\bar{v}$ with the following properties (informally):
	\begin{enumerate}
		\item $\bar{v}$ is $k$-sparse
		\item $\bar{v}$ is similar to the eigenvector $v^1$ relative to $\lambda_1$
		\item $\bar{v}$ is almost orthogonal to the eigenvectors of $M$ relative to its non-negative eigenvalues.
	\end{enumerate}
	Such vector gives a bound on $\lambda_1$ because using the eigendecomposition \eqref{eq:decomp}
	\begin{align*}
		0 \stackrel{\textrm{Obs \ref{obs:sparse}}}{\le} \bar{v}^\top M \bar{v} = - \sum_{i \le \ell} \lambda_i \, \ip{v^i}{\bar{v}}^2 + \sum_{i \le n-\ell} \mu_i \, \ip{w^i}{\bar{v}}^2 \lesssim -\lambda_1 + \textrm{``small error''},
	\end{align*}
	and hence $\lambda_1 \lesssim \textrm{``small error''}$.
	
	\medskip
	
	We show the existence of such $k$-sparse vector $\bar{v}$ via the probabilistic method by considering a random sparsification of $v^1$. More precisely, define the random vector $V \in \R^n$ as follows: in hindsight set $p:= 1 - \frac{2 (n-k)}{n}$, and let $V$ have independent entries satisfying
	$$
	V_i=\begin{cases}
	v^1_i & \text{ if }(v^1_i)^2>2/n, \\
	\frac{v^1_i}{p} \text{ with probability } p & \text{ if } (v^1_i)^2\leq \frac{2}{n}, \\
	0 \text{ with probability } 1-p & \text{ if } (v^1_i)^2\leq \frac{2}{n}.
	\end{cases}
	$$
	

	The choice of $p$ guarantees that $V$ is $k$-sparse with good probability.
	
	\begin{lemma} \label{lemma:sparse}
		$V$ is $k$-sparse with probability at least $\frac{1}{2}$.
	\end{lemma}
	
	\begin{proof}
	Let $m$ be the number of entries in $v^1$ with $(v^1_i)^2\leq \frac{2}{n}$. Since $\|v^1\|_2=1$ we have $m\geq \frac{n}{2}$. By the randomized construction, the number of coordinates of value 0 in $V$ is lower bounded by a binomial random variable $B$ with $m$ trials and success probability $1-p$. Using the definition of $p$ we have the expectation $$\E B = m (1-p) \ge \frac{n}{2} \cdot \frac{2(n-k)}{n} = n-k;$$ since $n-k$ is integer we have $\lfloor \E B \rfloor \ge n-k$. Moreover, it is known that the median of a binomial distribution is at least the expectation rounded down to the nearest integer~\cite{kaas1980mean}, hence $\Pr(B \ge \lfloor \E B \rfloor) \ge \frac{1}{2}$. Chaining these observations we have
	\begin{align*}
	\Pr\big(\text{\# of coordinates of value 0 in $V$} \ge n-k\big) \ge \Pr\big(B \ge n-k\big) \ge \Pr\big(B \ge \lfloor \E B \rfloor\big) \ge \frac{1}{2}.
	\end{align*}
	In other words, our randomized vector $V$ is $k$-sparse with probability at least $\frac{1}{2}$. 
	\end{proof}

 
	Next, we show that with good probability $V$ and $v_1$ are in a ``similar direction''.
	
	\begin{lemma}\label{lemma:parallel}
		With probability $> 1 - \frac{1}{6}$ we have $\ip{V}{v^1} \ge \frac{1}{2}$.
	\end{lemma}
	
	\begin{proof}
	To simplify the notation we use $v$ to denote $v^1$. By definition of $V$, for each coordinate we have $\E[V_i v_i] = v_i^2$, and hence $\E \ip{V}{v} = \|v\|_2^2 = 1$. 
	
	In addition, let $I$ be the set of coordinates $i$ where $v_i^2 \le \frac{2}{n}$. Then for $i \notin I$ we have $\var(V_i v_i) = 0$, and for $i \in I$ we have $\var(V_i v_i) = v_i^2 \var(V_i) \le \frac{2}{n} \var(V_i)$. Moreover, since $p \ge \frac{1}{2}$ (implied by the assumption $k \ge \frac{3n}{4}$) we have by construction  $V_i \le \frac{v_i}{p} \le 2 v_i$, and hence $$\var(V_i) \le \E V_i^2 \le 2 v_i \E V_i = 2 v_i^2.$$ So using the independence of the coordinates of $V$ we have $$\var \ip{V}{v} = \sum_{i \in I} \var(V_i v_i) \le \frac{4}{n}\, \sum_i v_i^2 = \frac{4}{n}.$$ Then by Chebyshev's inequality we obtain that
$$\Pr\left(\langle V, v\rangle \leq \frac{1}{2}\right) \leq \Pr \left( |\langle V, v\rangle - 1| \geq \frac{1}{2}\right) \leq \frac{16}{n}.$$ Since $n \geq 97$, this proves the lemma. 
	\end{proof}
	

	Finally, we show that $V$ is almost orthogonal to the eigenvectors of $M$ relative to non-negative eigenvalues. 
	
	\begin{lemma} \label{lemma:ortho}
		With probability $\ge 1 -\frac{1}{3}$ we have $\sum_{i \le n - \ell} \mu_i\,\ip{V}{w^i}^2 \le \frac{24(n-k)}{n^{3/2}}$.
	\end{lemma}
	
	\begin{proof}	
		Again we use $v$ to denote $v^1$.	Define the matrix $\overbar{M} := \sum_{i \le n - \ell} \mu_i w_i w_i^\top$, so we want to upper bound $V^\top \overbar{M} V$. Moreover, let $\Delta=V-v$; since $v$ and the $w_i$'s are orthogonal we have $\overbar{M} v = 0$ and hence 
	\begin{align}
		V^\top \overbar{M} V = v \overbar{M} v + 2 \Delta^\top \overbar{M} v + \Delta^\top \overbar{M} \Delta = \Delta^\top \overbar{M} \Delta, \label{eq:delta}
	\end{align}
	so it suffices to upper bound the right-hand side. 
		
	For that, notice that $\Delta$ has independent entries with the form 
	\begin{equation*}
	\Delta_i=\begin{cases}
	0 & \text{ if }v_i^2>\frac{2}{n}, \\
	\frac{v_i(1-p)}{p} \text{ with probability } p & \text{ if } v_i^2\leq \frac{2}{n}, \\
	-v_i \text{ with probability } 1-p & \text{ if } v_i^2\leq \frac{2}{n}.
	\end{cases}
	\end{equation*}
	So $\E [\Delta_i \Delta_j] = \E \Delta_i \E \Delta_j = 0$ for all $i\neq j$. In addition $\E \Delta_i^2 = 0$ for indices where $v^2_i > \frac{2}{n}$, and  
	\begin{align*}	
		\E \Delta_i^2 \le \frac{v_i^2 (1-p)^2}{p} + v_i^2 (1-p) = v_i^2 \frac{1-p}{p} \le \frac{2(1-p)}{np}.
	\end{align*}	
	Using these we can expand $\E[\Delta^T \overbar{M}\Delta]$ as 
	\begin{align}
	\E[\Delta^T \overbar{M}\Delta] = \E\bigg[\sum_{i,j} \overbar{M}_{ij}\Delta_i \Delta_j\bigg] = \sum_{i,j} \overbar{M}_{ij} \,\E[\Delta_i\Delta_j] &= \sum_{i=1}^{n} \overbar{M}_{ii}\,\E\Delta_i^2 \notag\\
	&\leq \frac{2(1-p)}{np}\Tr(\overbar{M}) \notag\\
	&= \frac{4 (n-k)}{n^2 p}\Tr(\overbar{M}), \label{eq:delta2}
	\end{align}
	where the last equation uses the definition of $p$. 
	
	Since the $\mu_i$'s are the eigenvalues of of $\overbar{M}$, we can therefore bound the trace as
	$$\Tr(\overbar{M})=\sum_{i \le n-\ell} \mu_i \leq \sqrt{n - \ell} \cdot \sqrt{\sum_{i \le n - \ell} \mu_i^2} \le \sqrt{n - \ell} \leq \sqrt{n},$$ where the first inequality follows from the well-known inequality that $\| u\|_1 \leq \sqrt{n}\|u\|_2$ for all $u \in \mathbb{R}^n$ and the second inequality uses $1 = \|M\|_F = \sqrt{\sum_{i \le \ell} \lambda_i^2 + \sum_{i \le n-\ell} \mu_i^2}$. Further using the assumption that $p \ge \frac{1}{2}$, we get from \eqref{eq:delta2} that $$\E[\Delta^T \overbar{M}\Delta] \le \frac{8 (n-k)}{n^{3/2}}.$$
	
	Finally, since all the eigenvalues $\mu_i$ of $\overbar{M}$ are non-negative, this matrix is PSD and hence the random variable $\Delta^\top \overbar{M} \Delta$ is non-negative. Markov's inequality 
	then gives that 
	\begin{align*}
		\Pr\bigg(\Delta^\top \overbar{M} \Delta \ge \frac{24 (n-k)}{n^{3/2}} \bigg) \le \Pr\bigg(\Delta^\top \overbar{M} \Delta \ge 3\, \E [\Delta^\top \overbar{M} \Delta] \bigg) \le \frac{1}{3}. 
	\end{align*}
	This concludes the proof of the lemma. 
	\end{proof}
	
	With these properties of $V$ we can finally upper bound the modulus $\lambda_1$ of the most negative eigenvalue of $M$.
	
	\begin{lemma} \label{lemma:lambda}
		$\lambda_1 \le \frac{96(n-k)}{n^{3/2}}.$
	\end{lemma}
	
	\begin{proof}
		We take the union bound 
		over Lemmas \ref{lemma:sparse} to \ref{lemma:ortho}. In other words, the probability that $V$ fails at least one of the properties in above three lemmas is strictly less than $\frac{1}{2}+\frac{1}{6}+\frac{1}{3}=1$. Therefore, with strictly positive probability $V$ satisfies all these properties. That is, there is a vector $\bar{v} \in \R^n$ that is $k$-sparse, has $\ip{\bar{v}}{v^1} \ge \frac{1}{2}$ and $\sum_{i \le n - \ell} \mu_i\,\ip{\bar{v}}{w^i}^2 \le \frac{24(n-k)}{n^{3/2}}$. Then using Observation \ref{obs:sparse} and the eigendecomposition \eqref{eq:decomp}
	\begin{align*}
		0 \stackrel{\textrm{Obs \ref{obs:sparse}}}{\le} \bar{v}^\top M \bar{v} = - \sum_{i \le \ell} \lambda_i \, \ip{\bar{v}}{v^i}^2 + \sum_{i \le n-\ell} \mu_i \, \ip{\bar{v}}{w^i}^2 \le -\frac{\lambda_1}{4} + \frac{24(n-k)}{n^{3/2}}.		
	\end{align*}
	Reorganizing the terms proves the lemma. 	
	\end{proof}


	\subsection{Concluding the proof of Theorem \ref{thm:upper2}}
	
Plugging the upper bound on $\lambda_1$ from Lemma \ref{lemma:lambda} into Proposition~\ref{prop:ub:smallev} we obtain that $$\dist_F(M, \psd) \le 96\, \bigg(\frac{n-k}{n}\bigg)^{3/2}.$$ Since this holds for all unit-norm $M \in \S^{n,k}$, we have that $\distb_F(\S^{n,k},\psd)$ also satisfies the same upper bound. This concludes the proof. 
	





\section{Proof of Theorem~\ref{thm:lower1}: A specific family of matrices in $\S^{n,k}$}\label{sec:proofThmlower1}
 To prove the lower bounds on $\distb_F(\mathcal{S}^{n,k},\mathcal{S}^n_+)$ we construct specific families of matrices in $\mathcal{S}^{n,k}$ with Frobenius norm 1, and then lower bound their distance to the PSD cone.
 
 For the first lower bound in Theorem~\ref{thm:lower1}, we consider the construction where all diagonal entries are the same, and all off-diagonal ones are also the same. More precisely, given scalars $a,b \ge 0$ we define the matrix 
	\begin{eqnarray}
		G(a,b, n):= (a + b) I_n -a \textbf{1} \textbf{1}^{\top},
	\end{eqnarray}
where $I_n$ is the $n \times n$ identity matrix, and $\textbf{1}$  is the column vector with all entries equal to $1$. In other words, all diagonal entries of $G(a,b,n)$ are $b$, and all off-diagonal ones are $-a$. 

The parameter $a$ will control how far this matrix is from PSD: for $a=0$ it is PSD, and if $a$ is much bigger than $b$ it should be ``far'' from the PSD cone. We then directly compute its eigenvalues, as well as its Frobenius distance to the PSD cone. 

\begin{proposition}\label{Gprop:evals}
	The eigenvalues of $G(a,b,n)$ are $b-(n-1)a$ with multiplicity 1, and $b+a$ with multiplicity $n-1$. 
\end{proposition}

\begin{proof}
	Let $\{v^1,...,v^n\}$ be an orthonormal basis of $\R^n$ such that $\sqrt{n}v^1=\textbf{1}$. Then we can rewrite $G(a,b,n)$ as 
	\begin{align*}
	G(a,b,n)& = (a+b)\sum_{i=1}^n v^i (v^i)^{\top}-nav^1 (v^1)^{\top}\\
	& = \big(b-(n-1)a\big)v^1 (v^1)^{\top}+(a+b)\sum_{i=2}^n v^i (v^i)^{\top}.
	\end{align*} 
	This gives a spectral decomposition of $G(a,b,n)$, so it has the aforementioned set of eigenvalues. 
\end{proof}

The next two corollaries immediately follow from Proposition~\ref{Gprop:evals}.
\begin{corollary}\label{cor:Gprop1}
	If $a,b \geq 0$, then $G(a, b, n) \in \mathcal{S}^{n,k}$ iff $b \geq (k -1)a$. In particular, since $\S^{n,n}=\psd$, $G(a,b,n)\in\psd$ iff $b\geq (n-1)a$. 
\end{corollary}
\begin{proof}
	Note that every $k \times k$ principal submatrix of $G(a, b, n)$ is just the matrix $G(a, b, k)$, which belongs to $\psd$ iff $b-(k-1)a\geq 0$, since $a,b\geq 0$. 
\end{proof}
\begin{corollary}\label{cor:Gprop2}
	If $a,b \geq 0$, then $\dist_F(G(a,b,n), \psd) = \max\{(n -1) a - b, 0\}$.
\end{corollary}
\begin{proof}
	If $b\geq (n-1)a$, then $G(a,b,n)\in\psd$ from first corollary, so $\dist_F(G(a,b,n), \psd)=0$ by definition. 
	
	If $b<(n-1)a$, then $G(a,b,n)$ has only one negative eigenvalue $b-(n-1)a$. Thus using Proposition~\ref{prop:neval} we get $\dist_F(G(a,b,n), \psd)=(n-1)a-b$. 
\end{proof}

	To conclude the proof of Theorem \ref{thm:lower1}, let $\bar{a} = \frac{1}{\sqrt{ ( k - 1)^2n  + n(n -1)} }$ and $\bar{b} = (k-1) \bar{a}$. From Corollary~\ref{cor:Gprop1} we know that $G(\bar{a}, \bar{b}, n)$ belongs to the \closure $\mathcal{S}^{n,k}$, and it is easy to check that it has Frobenius norm 1. Then using Corollary~\ref{cor:Gprop2} we get
	\begin{align*}
		\distb_F(\S^{n,k},\psd) \ge \dist_F(G(\bar{a} , \bar{b}, n), \psd) = (k-1) \bar{a} = \frac{n - k}{\sqrt{ ( k - 1)^2n  + n(n -1)}}.
	\end{align*}	
	This concludes the proof.

\section{Proof of Theorem~\ref{thm:lower2}: RIP construction when $k= O(n)$}\label{sec:proofThm3}

	Again, to prove the lower bound $\distb_F(\S^{n,k}, \psd) \ge cst$ for a constant $cst$ we will construct (randomly) a unit-norm matrix $M$ in the \closure $\S^{n,k}$ that has distance at least $cst$ from the PSD cone $\psd$; we will use its negative eigenvalues to assess this distance, via Proposition~\ref{prop:neval}. 	
	
	\paragraph{Motivation for connection with RIP property.} Before presenting the actual construction, we give the high-level idea of how the RIP property (Definition \ref{def:RIP} below) fits into the picture.  For simplicity, assume $k = n/2$.  (The actual proof will not have this value of $k$). The idea is to construct a matrix $M$ where about half of its eigenvalues take the negative value $-\frac{1}{\sqrt{n}}$, with orthonormal eigenvectors $v^1,v^2,\ldots, v^{n/2}$, and rest take a positive value $\frac{1}{\sqrt{n}}$, with orthonormal eigenvectors $w^1,w^2,\ldots, w^{n/2}$). This normalization makes $\|M\|_F = \Theta(1)$, so the reader can just think of $M$ being unit-norm, as desired. In addition, from Proposition~\ref{prop:neval} this matrix is far from the PSD cone: $\dist_F(M,\psd) \gtrsim \sqrt{\left(\frac{1}{\sqrt{n}}\right)^2 \cdot \frac{n}{2}} = cst$. So we only need to guarantee that $M$ belongs to the \closure; for that we need to carefully choose its positive eigenspace, namely the eigenvectors $w^1,w^2,\ldots, w^{n/2}$.
	
	Recall that from Observation \ref{obs:sparse}, $M$ belongs to the \closure iff $x^\top M x$ for all $k$-sparse vectors $x \in \R^n$. Letting $V$ be the matrix with rows $v^1,v^2,\ldots,$ and $W$ the matrix with rows $w^1,w^2,\ldots$, the quadratic form $x^\top M x$ is 
	
	\begin{align*}
		x^\top M x = - \frac{1}{\sqrt{n}} \sum_i \ip{v^i}{x}^2 + \frac{1}{\sqrt{n}} \sum_i \ip{w^i}{x}^2 = -\frac{1}{\sqrt{n}} \|Vx\|_2^2 + \frac{1}{\sqrt{n}} \|Wx\|_2^2.
	\end{align*}
	Since the rows of $V$ are orthonormal we have $\|Vx\|_2^2 \le \|x\|_2^2$. Therefore, if we \emph{could construct the matrix $W$ so that for all $k$-sparse vectors $x \in \R^n$ we had $\|Wx\|_2^2 \approx \|x\|_2^2$}, we would be in good shape, since we would have 
	\begin{align}
		x^\top M x \gtrsim - \frac{1}{\sqrt{n}} \|x\|_2^2 + \frac{1}{\sqrt{n}} \|x\|_2^2 \gtrsim 0 \qquad\qquad\textrm{for all $k$-sparse vectors $x$}, \label{eq:preRIP}
	\end{align}
	thus $M$ would be (approximately) in the \closure. This approximate preservation of norms of sparse vectors is precisely the notion of the \emph{Restricted Isometry Property} (RIP)~\cite{CandesTao1,candes2006stable}. 

\begin{definition}[RIP] \label{def:RIP}
	Given $k<m<n$, an $m\times n$ matrix $A$ is said to be $(k,\delta)$-RIP if for all $k$-sparse vectors $x \in \R^n$, we have
	$$
	(1-\delta)\|x\|_2^2\leq \|Ax\|_2^2\leq (1+\delta)\|x\|_2^2.
	$$
\end{definition}

		
This definition is very important in signal processing and recovery~\cite{Candes2008survey,Chartrand2008RIP,CandesTao1,candes2006stable}, and there has been much effort trying to construct deterministic~\cite{Calderbank2010RIPdet,Bandeira2013RIPdet} or randomized~\cite{baraniuk2008simple} matrices satisfying given RIP guarantees. 
	
	
	The following theorem in \cite{baraniuk2008simple} provides a probabilistic guarantee for a random Bernoulli matrix to have the RIP.

\begin{theorem}[(4.3) and (5.1) in \cite{baraniuk2008simple}] \label{thm:RIP}
	Let $A$ be an $m\times n$ matrix where each entry is independently $\pm 1/\sqrt{m}$ with probability $1/2$. Then $A$ is $(k,\delta)$-RIP with probability at least
	
	\begin{equation}\label{rip}
	1-2\left(\frac{12}{\delta}\right)^k e^{-\left(\delta^2/16-\delta^3/48\right)m}.
	\end{equation}
\end{theorem}
	

\paragraph{Proof of Theorem~\ref{thm:lower2}} After we have observed the above connection between matrices in $\mathcal{S}^{n,k}$ and RIP matrices, in the actual proof we adopt a strategy that does not ``flow'' exactly as described above but is easier to analyze. We will: 1) select $W$, a RIP matrix by selecting parameters $m$ and $\delta$ and applying Theorem~\ref{thm:RIP}; 2) use it to construct a matrix $M \in \mathcal{S}^{n, k}$; 3) rescale the resulting matrix so that its Frobenius norm is $1$, and; 4) finally compute its distance from $\psd$ and show that this is a constant independent of $n$.


	\paragraph{Actual construction of $M$.} Set $m = 93 k$ and $\delta = 0.9$. Then we can numerically verify that whenever $k\geq 2$, the probability \eqref{rip} is at least $0.51>\frac{1}{2}$. Then let $W$ be a random $m \times n$ matrix as in Theorem \ref{thm:RIP}, and define the matrix $$M := - (1-\delta) I + W^\top W.$$ 
	
	First observe that $M$ has a large relative distance to the PSD cone and with good probability belongs to the \closure.
	
	\begin{lemma} \label{lemma:RIP1}
		The matrix $M$ satisfies the following:
		\begin{enumerate}
			\item With probability at least $0.51$, $M \in \S^{n,k}$
			\item $\dist_F(M,\psd) \ge \sqrt{n-m}\,(1-\delta) $.
		\end{enumerate}
	\end{lemma}

	\begin{proof}	
		Whenever $W$ is $(k,\delta)$-RIP, by definition, for all $k$-sparse $x$ we have $x^{\top} W^{\top} Wx=\|Wx\|^2\geq (1-\delta)x^{\top }x$. Therefore $x^{\top} Mx\geq 0$ for all $k$-sparse $x$, and hence $M\in\Snk$ by Observation~\ref{obs:sparse}. This gives the first item of the lemma. 
		
		For the second item, notice that all vectors in the kernel of $W$, which has dimension $n-m$, are eigenvectors of $M$ with eigenvalue $-(1-\delta)$. So the negative eigenvalues of $M$ include at least $n-m$ copies of $-(1-\delta)$, and the result follows from Proposition~\ref{prop:neval}.
	\end{proof}	
	
	Now we need to normalize $M$, and for that we need to control its Frobenius norm. 

	
	\begin{lemma} \label{lemma:RIPnorm}
		With probability at least $\frac{1}{2}$, $\|M\|_F^2 \le  2n\delta^2+\frac{2n(n-1)}{m}$.
	\end{lemma}
	
	\begin{proof}
	Notice that the diagonal entries of $W^\top W$ equal $1$, so
$$
 \|M\|_F^2 = \sum_{i = 1}^n M_{ii}^2 + \sum_{i , j \in [n], i \neq j} M_{ij}^2 = n \delta^2 + \sum_{i , j \in [n], i \neq j} (W^{\top}W)_{ij}^2.
$$ 
We upper bound the last sum. Let the columns of $W$ be $C^1,...,C^n$, and denote by $X_{ij} = \ip{C^i}{C^j}$ the $ij$-th entry of $W^{\top} W$. Notice that when $i\neq j$, $X_{ij}$ is the sum of $m$ independent random variables $C^i_\ell C^j_\ell$ that take values $\{-\frac{1}{m},\frac{1}{m}\}$ with equal probability, where $\ell$ ranges from $1$ to $m$. Therefore,
$$
\E X_{ij}^2 = \var(X_{ij}) = \sum_{\ell \in [m]} \var(C^i_\ell C^j_\ell) = m \,\frac{1}{m^2} = \frac{1}{m}.
$$
This gives that 
$$
\E\, \|M\|_F^2 = n\delta^2+\frac{n(n-1)}{m}.
$$
Since $\|M\|_F^2$ is non-negative, from Markov's inequality $\|M\|_F^2 \le 2 \E \, \|M\|_F^2$ with probability at least $1/2$. This gives the desired bound, concluding the proof. 
	\end{proof}

	Taking a union bound over Lemmas \ref{lemma:RIP1} and \ref{lemma:RIPnorm}, with strictly positive probability the normalized matrix $\frac{M}{\|M\|_F}$ belongs to $\S^{n,k}$ and has $$\dist_F\left(\frac{M}{\|M\|_F}, \psd\right) \ge \frac{\sqrt{n-m}\,(1-\delta)}{\sqrt{2n(n-1)/m+2n\delta^2}} \geq \frac{\sqrt{n-m}\,(1-\delta)}{\sqrt{2n^2/m+2n\delta^2}}.$$ Thus, there is a matrix with such properties. 

Now plugging in $k=rn, m=93k, \delta=0.9$, the right hand side is  at least $\frac{\sqrt{r-93r^2}}{\sqrt{162r+3}}$. 
	This concludes the proof of Theorem \ref{thm:lower2}.


\section{Proof of Theorem~\ref{prop:sampling}}\label{sec:propsampling}

	The idea of the proof is similar to that of Theorem \ref{thm:upper1} (in Section \ref{sec:upper1}), with the following difference: Given a unit-norm matrix $M \in \S_{\cI}$, we construct a matrix $\overtilde{M}$ by averaging over the principal submatrices indexed by \emph{only the $k$-sets in $\cI$} instead of considering all $k$-sets, and upper bound the distance from $M$ to the PSD cone by $\dist_F(M, \alpha\overtilde{M})$. Then we need to provide a uniform upper bound on $\dist_F(M, \overtilde{M})$ that holds \emph{for all $M$'s simultaneously with good probability} (with respect to the samples $\cI$). This will then give an upper bound on $\distb_F(\S_{\cI},\psd)$.
	
	Recall that $\cI = (I_1,\ldots,I_m)$ is a sequence of independent uniform samples from the $k$-sets of $[n]$. 
	As defined in Section~\ref{sec:upper1}, let $M^I$ be the matrix where we zero out all the rows and columns of $M$ except those indexed by indices in $I$. Let $T_{\cI}$ be the (random) partial averaging operator, namely for every matrix $M \in \R^{n \times n}$ 
	\begin{align*}
		T_{\cI}(M) := \frac{1}{|\cI|} \sum_{I \in \cI} M^I.
	\end{align*}

	As we showed in Section \ref{sec:upper1} for the full average $\overtilde{M} := T_{{[n] \choose k}}(M)$, the first observation is that if $M \in \S_{\cI}$, that is, all principal submatrices $\{M_I\}_{I \in \cI}$ are PSD, then the partial average $T_{\cI}(M)$ is also PSD.
	
	\begin{lemma}\label{lemma:partialPSD}
		If $M \in \S_{\cI}$, then $T_{\cI}(M)$ is PSD. 
	\end{lemma}
	\begin{proof}
		This is straightforward, since each $M^I$ is PSD. 
	\end{proof}

Consider a unit-norm matrix $M$. Now we need to upper bound $\dist_F(M, \alpha\,T_{\cI}(M))$, for a scaling $\alpha$, in a way that is ``independent'' of $M$. In order to achieve this goal, notice that $(T_\cI(M))_{ij} = f_{ij} M_{ij}$, where $f_{ij}$ is the fraction of sets in $\cI$ that contain $\{i,j\}$. Then it is not difficult to see that the Frobenius distance between $M$ and $T_{\cI}(M)$ can be controlled using only these fractions $\{f_{ij}\}$, since the Frobenius norm of $M$ is fixed to be 1. 
	
	The next lemma makes this observation formal. Since the fractions $\{f_{ij}\}$ are random (they depend on $\cI$), the lemma focuses on the typical scenarios where they are close to their expectations. 
	
	Notice that the probability that a fixed index $i$ belongs to $I_\ell$ is $\frac{k}{n}$, so the fraction $f_{ii}$ is $\frac{k}{n}$ in expectation. Similarly, the expected value of $f_{ij}$ is $\frac{k(k-1)}{n(n-1)}$ when $i\neq j$. In other words, the expectation of $T_{\cI}(M)$ is $\overtilde{M}$. 
	
	
	\begin{lemma} \label{lemma:typical}
		Consider $\e \in [0,1)$ and let $\gamma := \frac{k(n-k)}{2n(n-1)}$. Consider a scenario where $\cI$ satisfies the following for some $\e \in [0,1)$:
		\begin{enumerate}
			\item For every $i \in [n]$, the fraction of the sets in $\cI$ containing $i$ is in the interval 
				$\left[\frac{k}{n} - \e \gamma, \frac{k}{n} + \e \gamma \right]$. 			
			\item For every pair of distinct indices $i,j \in [n]$, the fraction of the sets in $\cI$ containing both $i$ and $j$ is in the interval $\left[\frac{k(k-1)}{n(n-1)} - \e \gamma, \frac{k(k-1)}{n(n-1)} + \e \gamma \right]$.
		\end{enumerate}
		Then there is a scaling $\alpha > 0$ such that for all matrices $M \in \R^{n \times n}$ we have $$\dist_F(M, \alpha\, T_{\cI}(M)) \le (1+\e) \frac{n-k}{n+k-2}\,\|M\|_F.$$
	\end{lemma}

	\begin{proof}
	As in Section \ref{sec:upper1}, let $\overtilde{M} = T_{{[n] \choose k}}(M)$ be the full average matrix. Recall that $\overtilde{M}_{ii} = \overtilde{f}_{ii} M_{ii}$ for $\overtilde{f}_{ii} = \frac{k}{n}$, and $\overtilde{M}_{ij} = \overtilde{f}_{ij} M_{ij}$ for $\overtilde{f}_{ij} = \frac{k (k-1)}{n (n-1)}$ when $i \neq j$. Also let $\alpha := \frac{2n(n-1)}{k (n + k -2)}$. Finally, define $\Delta := \tilde{M} - T_{\cI}(M)$ as the error between the full and partial averages. 
	
	From triangle inequality we have
	\begin{align*}
		\|M - \alpha\, T_{\cI}(M)\|_F \le \|M - \alpha\, \overtilde{M}\|_F + \alpha\,\|\Delta\|_F. 
	\end{align*}
	Moreover, in Section \ref{sec:upper1} we proved the full average bound $\|M - \alpha\, \overtilde{M}\|_F \le \frac{n-k}{n+k-2} \|M\|_F$. Moreover, from our assumptions we have $f_{ij} \in [\overtilde{f}_{ij} - \e \gamma, \overtilde{f}_{ij} + \e \gamma]$ for all $i,j$, and hence $|\Delta_{ij}| \le \e \gamma\, |M_{ij}|$; this implies the norm bound $\|\Delta\|_F \le \e \gamma \|M\|_F$. Putting these bounds together in the previous displayed inequality gives 
	\begin{align*}
		\|M - \alpha\, T_{\cI}(M)\|_F \le  \bigg(\frac{n-k}{n+k-2} + \e \alpha \gamma\bigg)\,\|M\|_F = (1+\e) \frac{n-k}{n+k-2}\,\|M\|_F.
	\end{align*}
	This concludes the proof.
 	\end{proof}

	Finally, we use concentration inequalities to show that the ``typical'' scenario assumed in the previous lemma holds with good probability.
	
	\begin{lemma}
		With probability at least $1 - \delta$ and the parameter $m$ given in Theorem~\ref{prop:sampling}, the sequence $\cI$ is in a scenario satisfying the assumptions of Lemma \ref{lemma:typical}.
	\end{lemma}	
	
	\begin{proof}
	



	As stated in Lemma~\ref{lemma:typical}, we only need that for all entries $i,j$ the fraction $f_{ij}$ deviates from its expectation by at most $+\e\gamma$, with failure probability at most $\delta$. From union bound, this can be achieved if for each entry, the probability that the deviation of its fraction $f_{ij}$ fails to be within $[-\e\gamma, \epsilon\gamma]$ is at most $\frac{\delta}{n^2}$. Now we consider both diagonal and off-diagonal terms: 

\begin{enumerate}
	\item Diagonal terms $f_{ii}$: For each $k-$set sample $I$, let $X_I$ be the indicator variable that is 0 if $i\notin I$, and 1 if $i\in I$. Notice that they are independent, with expectation $\frac{k}{n}$. Let $X=\sum_{i\in\cI} X_I$ be the sum of these variables. 
	
	From definition of $f_{ii}$ we have $X=f_{ii}m$, where $m$ is the total number of samples. From Chernoff bound, have that 
	\begin{align*}
	\Pr\bigg(\left|f_{ii}-\frac{k}{n}\right|> \e\frac{(n-k)k}{2n(n-1)}\bigg) &= \Pr\bigg(\left|X-\frac{mk}{n}\right|> \e m\frac{(n-k)k}{2n(n-1)}\bigg)\\
	&\le 2\exp\bigg(-\frac{\e^2 (n-k)^2k m}{12n(n-1)^2 }\bigg)\\
	&\le \frac{\delta}{n^2} 
	\end{align*}
	as long as  $$m\geq \frac{12n(n-1)^2}{\e^2 (n-k)^2 k}\ln\frac{2n^2}{\delta}.$$
	
	\item Off-diagonal terms $f_{ij}$: Similar to first case, now for each $k-$set sample $I$, let $X_I$ be the indicator variable that is 1 if $\{i,j\}\subset I$, and 0 otherwise. Now the expectation of each $X_I$ becomes $\frac{k(k-1)}{n(n-1)}$. Again let $X=\sum_{i\in\cI} X_I$. 
	
	Using same argument as above, $X=f_{ij}m$. From Chernoff bound we get
	\begin{align*}
	\Pr\bigg(\left|f_{ij}-\frac{k(k-1)}{n(n-1)}\right|> \e\frac{(n-k)k}{2n(n-1)}\bigg) &= \Pr\bigg(\left|X-\frac{mk(k-1)}{n(n-1)}\right|> \e m\frac{(n-k)k}{2n(n-1)}\bigg)\\
	&\le 2\exp\bigg(-\frac{\e^2 (n-k)^2k m}{12n(n-1)(k-1) }\bigg)\\
	&\le \frac{\delta}{n^2}
	\end{align*}
	as long as $$m\geq \frac{12n(n-1)(k-1)}{\e^2 (n-k)^2 k}\ln\frac{2n^2}{\delta}.$$
	\end{enumerate}

	Since we chose $m$ large enough so it satisfies both of these cases, 
	taking a union bound over all $i,j$'s we get that the probability that any of the $f_{ij}$'s is $+\e\gamma$ more than their expectations is at most $\delta$. This concludes the proof. 
	\end{proof}

Combining this with Lemma~\ref{lemma:typical}, we conclude the proof of Theorem~\ref{prop:sampling}. \\
	
	%
	
\noindent \textbf{Acknowledgements.} Grigoriy Blekherman was partially supported by NSF grant DMS-1901950.

\bibliography{QCQP}{}
\bibliographystyle{siam}

\end{document}